\documentclass[a4paper,russian,english]{article}
\usepackage{amsmath, amsfonts, amssymb}
\usepackage[cp1251]{inputenc}
\usepackage[T2A]{fontenc}
\usepackage{babel,indentfirst}
\usepackage{graphicx,graphics,verbatim}
\usepackage{makeidx,graphicx}
\usepackage{amsthm,amsmath,amscd}   % ?????????????? ?????????? ?? AMS
\usepackage{amsfonts,amssymb}       % ?????????????? ?????????? ?? AMS
\usepackage{mathrsfs}
\usepackage{cite}  
\usepackage{xcolor}
\usepackage{lipsum}
\usepackage[colorlinks=true,linkcolor=teal,citecolor=blue,filecolor=magenta,urlcolor=cyan]{hyperref}
\usepackage{indentfirst} % ??????? ?????? ????? ?????????
\usepackage{geometry} % ?????? ???? ????????
\newtheorem{theorem}{Theorem}
\newtheorem{corollary}{Corollary}
\newtheorem{lemma}{Lemma}

\begin{document}
	
		\begin{center}
		\textbf {\Large Construction of Reproducing Kernel Hilbert Spaces\\
			Associated with the Laplace--Bessel Operator}
	\end{center}	
	\vspace{0pt plus1fill} %????? ????? fill = ????????? ???????????? ?????????? ?????????? fill, ??????? ???????? ????????? ?????? ?????
		\begin{center}
		{\large {\bf  E. L. Shishkina}}
	\end{center}

	\begin{center}  
 E-mail: ilina\_dicor@mail.ru\\

{Department of Mathematical and Applied Analysis, Voronezh State University, Voronezh,  Universitetskaya pl. 1, 394018 Russia}

{Department of Applied Mathematics and Computer Modeling, Belgorod State National Research University (BelGU), Pobedy St., 85, Belgorod, 308015, Russia}

{Institute of Mathematics, Physics and Information Technology, Kadyrov Chechen State University, A. Sheripova st., 32, Grozny, 364024, Russia}

{International Laboratory of Stochastic Analysis and its Applications, National Research University Higher School of Economics, Pokrovsky Bulvar, 11, Moscow, 109028, Russia}
 	\end{center}

{\bf Abstract:} {\small 
Motivated by applications to the study of  problems for partial differential equations with Bessel operator, we
introduce a new  reproducing kernel
Hilbert spaces (RKHS) which is fractional weighted Sobolev space
with  positive definite kernel. As an application, we  obtain a simple and practical approximate solution to the singular Poisson equation.}


{\bf 2010 Mathematics Subject Classification:} \texttt{46E22, 35J05}

\section{Introduction}

The theory of reproducing kernels, was popularized by N.~Aronszajn \cite{Aronszajn}, has found numerous applications in analysis, approximation theory, and inverse problems. In particular, Saitoh and his collaborators developed a systematic method for constructing approximate solutions to various inverse problems using RKHS and Tikhonov regularization \cite{Saitoh1997,SaitohSawano,ByunSaitoh,Matsuura}.

  We  use generalised Bessel potential spaces which are generalised fractional Sobolev spaces defined via the Hankel multiplier $(1+|\xi|^2)^{-\frac{\alpha}{2}}$. The Bessel potential spaces introduced by N.~Aronszajn and K.T.~Smith in \cite{Aronszajn001} have become fundamental objects in harmonic analysis and PDE theory. The space of generalized Bessel potentials
constructed using the Hankel transform, was first introduced by L.N.~Lyakhov and M.V. Polovinkina in \cite{LyakhovPolovinkina} via the Stein--Lizorkin approach. In that work, the norm in the space of generalized Bessel potentials  was constructed using the $B$-hypersingular integrals and Riesz $B$-potentials previously introduced by L.N.~Lyakhov in \cite{LyakhovHyp,LyakhovRiesz}. In the present paper, we adopt a different approach to defining the norm, based on the work of A.V.~Balakrishnan \cite{Balakrishnan}.

In this paper we prove that generalized Bessel potential space is a reproducing kernel Hilbert space  adapted to the Laplace--Bessel operator 
$$
\Delta_\gamma=\sum\limits_{i=1}^{n} \frac{\partial ^{2} }{\partial x_{i}^{2} } +\frac{\gamma _{i} }{x_{i} } \frac{\partial}{\partial x_i}.
$$
Next, we prove that the operator $\Delta_\gamma$ is bounded from the generalized Bessel potential space to the weighted Lebesgue space if and only if the smoothness parameter is at least two. Third, we derive an explicit formula for the regularized reproducing kernel, which is expressed in terms of the original reproducing kernel of the generalized Bessel potential space. Finally, we apply the theorems from \cite{Matsuura} to obtain a simple and practical approximate solution to the singular Poisson equation $\Delta_\gamma u{\,=\,}g$ and prove that the regularized solution converges uniformly to the exact solution as the regularization parameter tends to zero.

\section{Definitions}

\subsection{Laplace--Bessel operator and spaces}

The Bessel operator of the form
$$
(B_{\nu})_{t}=\frac{d^2}{dt^2}+\frac{\nu}{t}\frac{d}{dt},\qquad t>0,\qquad \nu>0,
$$
arises as the radial part of the Laplace operator when considering the Euclidean distance, as part of the wave operator in the context of the Lorentz distance, and also in various applied problems.

The Laplace--Bessel operator is given by $\Delta_\gamma{\,=\,}\sum\limits_{i=1}^{n} (B_{\gamma_i})_{x_i}.$

When considering the action of the Laplace--Bessel operator $\Delta_\gamma$ on functions of several variables, it is usually assumed that these variables cannot take negative values. We therefore work in the $n$-dimensional Euclidean space $\mathbb{R}^n$ and its open orthant
$$
\mathbb{R}^n_+=\{x=(x_1,\ldots,x_n)\in\mathbb{R}^n,\,\,\, x_1>0,\ldots, x_n>0\}.
$$
All functions considered below are extended to negative values of $x_i$, $i=1,\ldots,n$, by even reflection.

Let
$\gamma=(\gamma_1,\ldots,\gamma_{n})$ be a multi-index consisting of positive real numbers $\gamma_i>0$, $i=1,\ldots,n$, and let $|\gamma|=\gamma_1+\cdots+\gamma_{n}$.

The weighted Lebesgue space $L^{p}_\gamma(\mathbb{R}^{n}_+)$ consists of measurable functions $f$ on $\mathbb{R}^{n}_+$ with finite norm
$$
\|f\|_{p,\gamma}=\left(\,\int\limits_{\mathbb{R}^{n}_+}|f(x)|^{p}x^\gamma\,dx\right)^{1/p},\qquad 1\le p<\infty,
$$
where $ x^\gamma{\,=\,}\prod\limits_{i=1}^n x_i^{\gamma_i}$.

The space $L_\gamma^2(\mathbb{R}_+^n)$ is a Hilbert space consisting of functions on $\mathbb{R}_+$, with the inner product given by
$$
\langle u, v \rangle_{L_\gamma^2(\mathbb{R}_+^n)}=\langle u, v \rangle_\gamma
=
\int\limits_{\mathbb{R}_+^n}  u(x)v(x)x^\gamma\,dx.
$$ 
This inner product we will call  $L^2_\gamma$-inner product or weighted inner product.

Let $\Omega$ be a finite or infinite open set in $\mathbb{R}^n$ that is symmetric with respect to each hyperplane $x_i = 0$, $i = 1, \dots, n$. Let
$
\Omega_+{\,=\,} \Omega \cap \mathbb{R}^n_+
$
and
$
\overline{\Omega}_+{\,=\,}\Omega \cap \overline{\mathbb{R}}\,\!^n_+,
$
where
$$
\overline{\mathbb{R}}\,^n_+ = \{ x = (x_1, \dots, x_n) \in \mathbb{R}^n : x_1 \geq 0, \dots, x_n \geq 0 \}.
$$
We deal with the class $C^m(\Omega_+)$ consisting of functions that are $m$ times differentiable on $\Omega_+$, and denote by $C^m(\overline{\Omega}_+)$ the subset of functions from $C^m(\Omega_+)$ such that all derivatives of these functions with respect to $x_i$ for any $i=1,\dots,n$ are continuous up to $x_i=0$.

The class $C^m_{ev}(\overline{\Omega}_+)$ consists of all functions from $C^m(\overline{\Omega}_+)$ such that
$$
\frac{\partial^{2k+1}f}{\partial x_i^{2k+1}}\bigg|_{x_i=0}=0
$$
for all non-negative integers $k \leq \frac{m-1}{2}$ (see \cite{Kipr}, p. 21).

In what follows, we will denote $C^m_{ev}(\overline{\mathbb{R}}\,\!^n_+)$ by $C^m_{ev}$.

We set
$$
C^\infty_{ev}(\overline{\Omega}_+)=\bigcap C^m_{ev}(\overline{\Omega}_+),
$$
where the intersection is taken over all finite $m$, and write $C^\infty_{ev}(\overline{\mathbb{R}}_+)=C^\infty_{ev}$.

As the space of basic functions we will use  the subspace of the space of rapidly decreasing functions:
$$
S_{ev}=\left\{f\in C^\infty_{ev}:\sup _{{x\in {\mathbb{R}}^{n}_+}}\left|x^{\alpha }D^{\beta }f(x)\right|<\infty \quad \forall \alpha ,\beta \in \mathbb {Z} _{+}^{n}\right\},
$$
where $\alpha=(\alpha_1,...,\alpha_n)$, $\beta=(\beta_1,...,\beta_n)$, $\alpha_1,...,\alpha_n,\beta_1,...,\beta_n$ are integer non-negative numbers, 
$x^\alpha{\,=\,}x_1^{\alpha_1} x_2^{\alpha_2} \ldots x_n^{\alpha_n}$,
${D}^\beta={D}^{\beta_1}_{x_1}...{D}^{\beta_n}_{x_n}$, ${D}_{x_j}=\frac{\partial}{\partial x_j}$.

The space of weighted generalized functions $S_{ev}'(\mathbb{R}^n_+){\,=\,}S_{ev}'$ is a class of continuous linear functionals that map a set of test functions  $f\in  S_{ev}$ into the set of real numbers. 
Each function $u(x){\,\in\,}L_{1,loc}^\gamma$ will be identified with the functional $u{\,\in\,}S_{ev}'(\mathbb{R}^n_+){\,=\,}S_{ev}'$
acting according to the formula
\begin{equation}\label{RegDist}
	\langle u,f\rangle_\gamma=\int\limits_{\mathbb{R}^n_+} u(x)\,f(x)\,x^\gamma\, dx,\qquad f\in  S_{ev}.
\end{equation}
Generalized functions $u\in S_{ev}'$ acting by the formula \eqref{RegDist} will be called  regular weighted generalized functions.
All other generalized functions  $u\in S_{ev}'$ will be called  singular weighted generalized functions.

The weighted delta-function $\delta_\gamma \in S'_{ev}$ is defined by the equality 
$$
\langle \delta_\gamma,\varphi\rangle_\gamma = \varphi(0), \qquad \varphi(x) \in S_{ev}.
$$
The action  of this generalized function can be understood as follows.
Let
$$
\omega_{\varepsilon}(x)=
\begin{cases}
	C_\varepsilon\, e^{-\frac{\varepsilon^2}{\varepsilon^2-|x|^2}}, & |x| \leq \varepsilon, \\[4pt]
	0, & |x| > \varepsilon,
\end{cases}
$$
where the constant $C_\varepsilon$ is chosen so that
$$
\int\limits_{\mathbb{R}^n_+} \omega_{\varepsilon}(x) \, x^\gamma \, dx = 1.
$$
Since
$$
\lim_{\varepsilon \to +0} \int\limits_{\mathbb{R}^n_+} \omega_{\varepsilon}(x) \, \varphi(x) \, x^\gamma \, dx = \varphi(0), \qquad \varphi \in S_{ev},
$$
we have
$$
\langle \omega_{\varepsilon}, \varphi\rangle_\gamma \to \langle\delta_{\gamma}, \varphi\rangle_\gamma, \qquad \varepsilon \to +0, \qquad \varphi \in S_{ev}.
$$
For convenience, we will write
$$
\langle \delta_\gamma, \varphi\rangle_\gamma = \int\limits_{\mathbb{R}^n_+} \delta_\gamma(x) \, \varphi(x) \, x^\gamma \, dx = \varphi(0),
$$
with the understanding that this expression is to be interpreted as the limit of a delta-shaped sequence.

Next we give the definition of a reproducing kernel Hilbert space from \cite{SaitohSawano}.

Let $E$ denote an arbitrary, non-empty abstract set, and let $\mathcal{F}(E)$ be the set of all complex-valued functions defined on $E$. A {\bf reproducing kernel Hilbert space} (abbreviated as RKHS) over $E$ consists of a Hilbert space $\mathcal{H} \subseteq \mathcal{F}(E)$ together with a function $K{\,:\,}E{\,\times\,}E {\,\to\,}\mathcal{H} $, termed the {\bf reproducing kernel}, which satisfies the following two conditions:
\begin{enumerate}
	\item For every $x \in E$, the function $K(\cdot, x)$ belongs to $\mathcal{H}$.
	\item For every $x \in E$ and every $f \in \mathcal{H}$, the evaluation at $x$ is given by the inner product
	$$
	f(x) = \langle f(x), K(\cdot, x) \rangle_{\mathcal{H}}.
	$$
\end{enumerate}

The Hilbert space $\mathcal{H}$ is denoted by $\mathcal{H}_K(E)$ (or simply $\mathcal{H}_K$) when it is associated with its reproducing kernel $K$. 
The correspondence of the reproducing kernel $K$ and the reproducing 	kernel Hilbert space $\mathcal{H}_K(E)$ is one-to-one.

\subsection{Multidimensional Hankel transform and Multidimensional generalized translation}

The Bessel function of the first kind $J_\alpha(x)$ for non-integer $\alpha$ is defined by the power series \cite{Watson}
\begin{equation}\label{J1}
	J_\alpha(x) = \sum\limits_{m=0}^\infty \frac{(-1)^m}{m!\, \Gamma(m+\alpha+1)} {\left({\frac{x}{2}}\right)}^{2m+\alpha}.
\end{equation}
For integer $\alpha$, the definition is obtained by taking the limit and applying L'H\^opital's rule.

The modified Bessel function of the first kind $I_{\alpha }(x)$ is defined by \cite{Watson}
\begin{equation}\label{Mod1}
	I_{\alpha }(x)=i^{-\alpha }J_{\alpha }(ix)=\sum _{m=0}^{\infty }{\frac {1}{m!\,\Gamma (m+\alpha +1)}}\left({\frac {x}{2}}\right)^{2m+\alpha }.
\end{equation}

The normalized Bessel function of the first kind $j_\nu$ is given by (see \cite{Kipr}, \cite{levitan})
$$
	j_\nu(x) =\frac{2^\nu\Gamma(\nu+1)}{x^\nu}\,\,J_\nu(x),
$$
where $J_\nu$ is the Bessel function of the first kind \eqref{J1}.
For $\nu=\frac{\gamma-1}{2}$, $\gamma>0$, we have
\begin{equation}\label{FBess1}
	j_\frac{\gamma-1}{2}(x) =\frac{2^\frac{\gamma-1}{2}\Gamma\left(\frac{\gamma+1}{2}\right)}{x^\frac{\gamma-1}{2}}\,\,J_\frac{\gamma-1}{2}(x).
\end{equation}

The normalized modified Bessel function of the first kind $i_\nu$ is defined by
\begin{equation}\label{FBess2}
	i_\nu(x) =\frac{2^\nu\Gamma(\nu+1)}{x^\nu}\,\,I_\nu(x),
\end{equation}
where $I_\nu$ is the modified Bessel function of the first kind \eqref{Mod1}.

Let $\gamma=(\gamma_1,\ldots,\gamma_n)$ with $\gamma_1>0,\ldots,\gamma_n>0$, and define
\begin{equation}\label{Pr1}
	\mathbf{j}_\gamma(x,\xi)=\prod\limits_{k=1}^n j_{\frac{\gamma_k-1}{2}}(x_k\xi_k),
\end{equation}
\begin{equation}\label{Pr2}
	\mathbf{i}_\gamma(x,\xi)=\prod\limits_{k=1}^n i_{\frac{\gamma_k-1}{2}}(x_k\xi_k),
\end{equation}
where the functions $j_\nu$ and $i_\nu$ are given by \eqref{FBess1} and \eqref{FBess2}, respectively.

It is known (see \cite{ShS}), that
	\begin{equation}\label{DelBes}
\Delta_\gamma \mathbf{j}_\gamma(x,\xi)=-|\xi|^2\mathbf{j}_\gamma(x,\xi)
\end{equation}
and
	\begin{equation}\label{JBes}
	\int\limits_{S_1^+(n)}\mathbf{j}_\gamma(r\theta,\xi) \theta^\gamma\:dS=\frac{\prod\limits^n_{i=1}\Gamma\left(\frac{\gamma_i+1}{2}\right)}{2^{n-1}\Gamma\left(\frac{n+|\gamma|}{2}\right)}\, j_{\frac{n+|\gamma|}{2}-1}(r|\xi|),
\end{equation}
where $\theta^\gamma{=}\prod\limits_{i=1}^{n}\theta_i^{\gamma_i},$ $S^+_1(n){=}\{\theta{:}|\theta|{=}1,\theta{\in}\mathbb{R}^n_+\}$ is a part of a sphere in
$\mathbb{R}^n_+$.

The eigenfunctions of $\Delta_\gamma$ are products of normalized Bessel functions \eqref{Pr1}, \eqref{Pr2}. Functions \eqref{Pr1} are used to define the multidimensional Hankel transform. The use of these eigenfunctions has enabled the development of a new harmonic analysis associated with the operator $\Delta_\gamma$.

The multidimensional Hankel transform of a function $f\in L^1_\gamma(\mathbb{R}_+^n)$ is defined by
$$
\mathbf{F}_\gamma[f](\xi)=(\mathbf{F}_\gamma)_x[f(x)](\xi)=\widehat{f}(\xi)=\frac{2^{\frac{n-|\gamma|}{2}}}{\prod\limits_{j=1}^n\,
	\Gamma\left(\frac{\gamma_j+1}{2}\right)}\int\limits_{\mathbb{R}^n_+}f(x)\,\mathbf{j}_\gamma(x;\xi)x^\gamma dx.
$$

The subspace of the Schwartz space $S(\mathbb{R}^{n})$ consisting of functions that are even in each variable will be denoted by $S_{ev}(\mathbb{R}^{n}_+)$.
For $f{\,\in\,}S_{ev}(\mathbb{R}_+^n)$, the inverse multidimensional Hankel transform is given by
$$
\mathbf{F}^{-1}_\gamma[\widehat{f}(\xi)](x)=f(x)=\frac{2^{\frac{n-|\gamma|}{2}}}{\prod\limits_{j=1}^n\,
	\Gamma\left(\frac{\gamma_j+1}{2}\right)}\int\limits_{\mathbb{R}^n_+}
\mathbf{j}_\gamma(x,\xi)\widehat{f}(\xi)\xi^\gamma\:d\xi.
$$

 The  multidimensional Hankel transform defined this way is also its own inverse: $\mathbf{F}_\gamma{\,=\,}\mathbf{F}^{-1}_\gamma$.

For functions $f, g \in L^2_\gamma(\mathbb{R}_+^n)$, the Parseval's identity states that
\begin{equation}\label{Parseval}
\int\limits_{\mathbb{R}_+^n} f(x) \, \overline{g(x)} \, x^\gamma \, dx
	=
	\int\limits_{\mathbb{R}_+^n} \mathbf{F}_\gamma[f](\xi) \, \overline{\mathbf{F}_\gamma[g](\xi)} \, \xi^\gamma \, d\xi,
\end{equation}
where $\overline{g}$ denotes the complex conjugate of $g$.

Taking $f{\,=\,}g$ in \eqref{Parseval} yields the Plancherel's identity:
\begin{equation}\label{Plancherel}
	\int\limits_{\mathbb{R}_+^n} |f(x)|^2 \, x^\gamma \, dx
	=
	\int\limits_{\mathbb{R}_+^n} |\mathbf{F}_\gamma[f](\xi)|^2 \, \xi^\gamma \, d\xi
\end{equation}
or equivalently,
$$
	\|f\|_{L^2_\gamma(\mathbb{R}_+^n)} = \|\mathbf{F}_\gamma[f]\|_{L^2_\gamma(\mathbb{R}_+^n)}.
$$
Therefore, the multidimensional Hankel transform $\mathbf{F}_\gamma$ is a unitary operator on $L^2_\gamma(\mathbb{R}_+^n)$,
preversing the weighted $L^2_\gamma$-inner product.

	For the Laplace--Bessel operator, the Hankel transform gives
\begin{equation}\label{LBH}
\mathbf{F}_\gamma[\Delta_\gamma f](\xi) = -|\xi|^2 \mathbf{F}_\gamma[f](\xi).
\end{equation}

	If $g\in S_{ev}'$ then equality
$$
	\langle\mathbf{F}_\gamma g,\varphi\rangle_\gamma=\langle g,\mathbf{F}_\gamma\varphi\rangle_\gamma,\qquad\varphi\in S_{ev}
$$
defines  Hankel transform of
functional $g\in S_{ev}'$.

%\subsection{Multidimensional generalized translation}

The generalized translation operator $\,^\gamma T^y_x$ associated with the Bessel operator is defined by
\begin{equation}\label{Sdvog011}
	(\,^\gamma T^y_xf)(x)=\,^\gamma T^y_xf(x)=\frac{\Gamma\left(\frac{\gamma+1}{2}\right)}{\sqrt{\pi}\,\,\Gamma\left(\frac{\gamma}{2}\right)}\int\limits_0^\pi
	f(\sqrt{x^2+y^2-2xy\cos{\varphi}})\sin^{\gamma-1}{\varphi}\,d\varphi.
\end{equation}

The multidimensional generalized translation is defined by
$$
	(^\gamma \mathbf{T}^y_x f)(x)=\,^\gamma \mathbf{T}^y_x f(x)=(\,^{\gamma_1}T_{x_1}^{y_1}{\cdots}\,^{\gamma_n}T_{x_n}^{y_n}f)(x),
$$
where each of the generalized shifts $\,^{\gamma_i}T_{x_i}^{y_i}$, $i=1,\ldots,n$, is given by \eqref{Sdvog011}.

The next equation represents a basic property of the multidimensional generalized translation
$$
	\,^\gamma \mathbf{T}^y_x\, \mathbf{j}_\gamma(x;\xi)=\mathbf{j}_\gamma(x;\xi)\mathbf{j}_\gamma(y;\xi).
$$

	Hankel transform from generalized translation of function $f{\in}S_{ev}(\mathbb{R}_+)$ has a form
\begin{equation}\label{HanSdv}
	\mathbf{F}_\gamma[ \,^\gamma \mathbf{T}^y_xf(x)](\xi)=\mathbf{j}_{\gamma}(y,\xi)\,\mathbf{F}_\gamma[f](\xi).
\end{equation}

The generalized convolution associated with the multidimensional generalized translation $\,^\gamma \mathbf{T}^y_x$ is defined by
\begin{equation}\label{Conv}
	(f*g)_\gamma(x)=(f*g)_\gamma=\int\limits_{\mathbb{R}^n_+}f(y)(^\gamma\mathbf{T}^y_xg)(x)y^\gamma\:dy.
\end{equation}

For $f,g\in S_{ev}(\mathbb{R}_+^n)$, the multidimensional Hankel transform of the generalized convolution \eqref{Conv} satisfies
\begin{equation}\label{fursver1}
	\mathbf{F}_\gamma[(f*g)_\gamma(x)](\xi)=\mathbf{F}_\gamma[f(x)](\xi)\mathbf{F}_\gamma[g(x)](\xi).
\end{equation}

	It is clear that the Hilbert space $L^{p}_\gamma(\mathbb{R}^{n}_+)$ is  not an RKHS because the  function $(\,^\gamma \mathbf{T}^y_x\delta_\gamma)(x)$ which has the
	reproducing property
$$
\langle \,^\gamma \mathbf{T}^y\delta_\gamma, \varphi\rangle_\gamma = \int\limits_{\mathbb{R}^n_+} (\,^\gamma \mathbf{T}^y_x\delta_\gamma)(x) \, \varphi(x) \, x^\gamma \, dx = \varphi(y),
$$
	does not satisfy the square integrable condition with weight $x^\gamma$.

\section{Semigroup Approach to Generalised Bessel Potential Operator}
	
	\subsection{Singular Rescaled Semigroup}
	
The semigroup generated by the operator $((\Delta_\gamma)_{x}-cI)$, with  $c{\,>\,0}$ 
is referred to as a {\bf singular rescaled  semigroup} (or exponentially damped semigroup). An explicit construction can be obtained by combining the fundamental solution of the singular heat equation with a scalar multiplication operator.

	Let  $g(x)$   be a uniformly bounded, continuous function defined on 
 $\mathbb{R}^n_+$, $c{\,\in\,\mathbb{R}}$.	We need to focus on solving the following problem for the {\bf singular diffusion equation with absorption} with $u{\,=\,}u(x,t)$, $x{\,\in\,}\mathbb{R}^n_+$, $t{\,\geq\,}0$. 
		\begin{equation}\label{ZCauch01}
		\begin{cases}
			\dfrac{\partial u}{\partial t} = (\Delta_\gamma)_{x} u - c u, & x \in \mathbb{R}^n_+, \quad t > 0; \\[3pt]
			u(x,0) = g(x), &  x \in \mathbb{R}^n_+; \\[3pt]
			\left.\dfrac{\partial u}{\partial x_i}\right|_{x_i=0}=0,& i=1,...,n.
		\end{cases}
	\end{equation}	
	We notice that if $v{\,=\,}v(x,t)$ is a solution of the above equation in \eqref{ZCauch01} with $c{\,=\,} 0$ (i.e., $v$ solves the singular heat equation), then setting
	$$
	u(x,t) = e^{-ct} v(x,t)
	$$
	we find that $v$ has the same initial condition as $u$, and
	$$
\dfrac{\partial u}{\partial t}  = -c e^{-ct}v  + e^{-ct}\dfrac{\partial v}{\partial t} , \qquad
	(\Delta_\gamma)_{x} u  = e^{-ct} (\Delta_\gamma)_{x} v.
	$$
	Therefore,
	$$
\dfrac{\partial u}{\partial t} - (\Delta_\gamma)_{x} u + c u
	= c e^{-ct}v + e^{-ct}\dfrac{\partial v}{\partial t} - e^{-ct}a^2(\Delta_\gamma)_{x} v + c e^{-ct} v
	= 0,
$$
	by virtue of the fact that 
	$$
\dfrac{\partial v}{\partial t}=(\Delta_\gamma)_{x} v.
$$
	
	However, we already know the formula for $v$ (see \cite{ZhVMSh}):
	$$
	v(x,t)=\frac{t^{-\frac{n+\lvert\gamma\rvert }{2}}}{2^{\lvert\gamma\rvert }\prod\limits_{i=1}^n{\Gamma\left(\frac{\gamma_i{+}1}{2}\right)}}\int\limits_{\mathbb{R}^n_+}{e^{-\frac{\lvert y\rvert ^2}{4t}}} (^\gamma\mathbf{T}^y_x\varphi)(x)y^\gamma dy.
	$$
	The function $v(x,t)$ thus defined is continuous and bounded for all $t{\,\geq\,} 0$.
	Therefore, the continuous and bounded function
	$$
	u(x,t) = \frac{e^{-ct}t^{-\frac{n+\lvert\gamma\rvert }{2}}}{2^{\lvert\gamma\rvert }\prod\limits_{i=1}^n{\Gamma\left(\frac{\gamma_i{+}1}{2}\right)}}\int\limits_{\mathbb{R}^n_+}{e^{-\frac{\lvert y\rvert ^2}{4t}}} (^\gamma\mathbf{T}^y_x\varphi)(x)y^\gamma dy.
	$$
	solves the desired  problem \eqref{ZCauch01}.

For a   function $\varphi{\,\in\,}L^p_\gamma(\mathbb{R}^n_+)$, $1\leq p<\infty$, singular rescaled  semigroup $\{S^\gamma_t\}_{t\geq 0}$ applied to $g$ at time $t{\,>\,}0$ is explicitly given as:
	$$
	(S^\gamma_t \varphi)(x)=\frac{e^{-ct}t^{-\frac{n+|\gamma| }{2}}}{2^{|\gamma| }\prod\limits_{i=1}^n{\Gamma\left(\frac{\gamma_i{+}1}{2}\right)}}\int\limits_{\mathbb{R}^n_+}{e^{-\frac{|y|^2}{4t}}} (^\gamma\mathbf{T}^y_x\varphi)(x)y^\gamma dy.
	$$

From \cite{ZhVMSh} follows that for $c{\,>\,}0$
the operators $\{S^\gamma_t\}_{t\geq 0}$ form a $C_0$-semigroup of contractions on $L^\gamma_p(\mathbb{R}^n_+)$ for $1\leq p<\infty$. Its generator $A$
 coincides with the closure of the operator $((\Delta_\gamma)_{x}-cI)$ defined on the space of even Schwartz functions $S_{ev}(\mathbb{R}^n_+)$.
	
Since (see \cite{AlSh})
$$
\,^\gamma \mathbf{T}^y_x e^{-\frac{| x|^2}{4 a^2 t}}=e^{-\frac{| x|^2+| y|^2}{4a^2t}}\mathbf{i}_{\gamma}\left(x,\frac{y}{2a^2t}\right),
$$
where $\mathbf{i}_\gamma$ is defined by equality \eqref{Pr2}, then
the operator $S^\gamma_t$ can be written in the form:
$$
(S^\gamma_t\varphi)(x)=\frac{e^{-ct}t^{-\frac{n+|\gamma| }{2}}}{2^{|\gamma| }\prod\limits_{i=1}^n{\Gamma\left(\frac{\gamma_i{+}1}{2}\right)}}\int\limits_{\mathbb{R}^n_+} \varphi(y)e^{-\frac{| x|^2+| y|^2}{4t}}\mathbf{i}_{\gamma}\left(x,\frac{y}{2t}\right) y^\gamma dy.
$$

The   norm of $S^\gamma_t$ in The norm decays exponentially as decays exponentially as
		$$
	||S^\gamma_t g||_{p,\gamma}\leq e^{-ct}||g||_{p,\gamma}.
	$$

	In the spatial domain, the action of the semigroup $S^\gamma_t$ on a function $g$ is given by the generalised 
	convolution of $g$ with a scaled Gaussian kernel, depending on $\gamma$, modulated by the exponential decay
$$
	W_{\gamma,c}(x,t)=\frac{1}{2^{|\gamma| }\prod\limits_{i=1}^n{\Gamma\left(\frac{\gamma_i{+}1}{2}\right)}}
	\left\{ \begin{array}{ll}
	t^{-\frac{n+|\gamma|}{2}}{e^{-\frac{|x|^2}{4t}-ct}}, & \mbox{if $t > 0$},\\
		0, & \mbox{if $t \leq 0$}.
\end{array} \right.
$$
Namely,
\begin{equation}\label{ConvS}
(S^\gamma_t\varphi)(x)=(W_\gamma*\varphi)_\gamma(x)=\int\limits_{\mathbb{R}^n_+}W_\gamma(y,t) (^\gamma\mathbf{T}^y_x\varphi)(x)y^\gamma dy.
\end{equation}

{\bf Properties of $W_\gamma(x,t)$.}

\begin{enumerate}
	\item The Hankel transform of the function $W_\gamma(x,t)$ with respect to $x\in\mathbb{R}^n_+$ and $t>0$ has the form (see \cite{AlSh})
	\begin{equation}\label{HT}
		(\mathbf{F}_\gamma)_x[W_\gamma(x,t)](\xi,t)=e^{-t(c+\lvert\xi\rvert ^2)}.
	\end{equation}
	\item The integral over the orthant of $W_\gamma(x,t)$ gives $e^{-ct}$
$$
		\int\limits_{\mathbb{R}^{n}_+}W_\gamma(x,t)x^\gamma dx=e^{-ct}.
$$
\end{enumerate}

Formulas \eqref{ConvS} and \eqref{HT} allow us to write the following representation of
 $S^\gamma_t$ when $\varphi\in S_{ev}(\mathbb{R}^n_+)$ 
\begin{equation}\label{SdExp}
(S^\gamma_t\varphi)(x)=\frac{2^{n-|\gamma|}}{\prod\limits_{j=1}^n\,
	\Gamma^2\left(\frac{\gamma_j{+}1}{2}\right)}\int\limits_{\mathbb{R}^n_+}[\mathbf{F}_\gamma\varphi](\xi) e^{-t(c+|\xi|^2)}\mathbf{j}_\gamma(x,\xi) \xi^\gamma\:d\xi=[\mathbf{F}_\gamma^{-1}([\mathbf{F}_\gamma\varphi](\xi) e^{-t(c+|\xi|^2)})](x).
\end{equation}

\subsection{Construction of the Fractional Power of the $(cI-\Delta_\gamma)$}

The semigroup method is a general technique that allows one to formulate and analyze the fundamental aspects of fractional powers of operators and their properties in appropriate function spaces. In this section, we demonstrate how this method works in the particular case of defining the fractional Laplace--Bessel operator. Our starting point is the Balakrishnan formula 
\begin{equation}\label{BF03}
	(-A)^{-\alpha} u=\frac{1}{\Gamma(\alpha)}\int\limits_0^\infty t^{\alpha-1}T_tudt,\qquad \alpha>0,
\end{equation}
which expresses the negative fractional power of an operator in terms of the semigroup it generates.

The idea of introducing the negative fractional power of the  operator  $(cI-\Delta_\gamma)$ is transparent in the Hankel transform domain:
$$
	(cI-\Delta_\gamma)^{-\frac{\alpha}{2}}\varphi=\mathbf{F}_\gamma^{-1}(c+|x|^2)^{-\frac{\alpha}{2}} \mathbf{F}_\gamma\varphi,\qquad \alpha>0.
$$

\begin{theorem}
 Let $\varphi \in S_{ev}(\mathbb{R}^n_+)$, $c{\,>\,}0$. 
	For $0{\,<\,}\alpha$, the negative fractional power of the operator $(cI-\Delta_\gamma)$ has the form
	\begin{equation}\label{FrL01}
		(cI-\Delta_\gamma)^{-\frac{\alpha}{2}} \varphi(x)
			=\frac{2^{\frac{n-|\gamma|-\alpha}{2}+1} c^{\frac{n+|\gamma|- \alpha}{4}} }{\Gamma\left(\frac{\alpha}{2}\right)\prod\limits_{i=1}^n{\Gamma\left(\frac{\gamma_i{+}1}{2}\right)}}  \int\limits_{\mathbb{R}^n_+}|y|^{\frac{\alpha-n-|\gamma|}{2}}K_{\frac{n+|\gamma|-\alpha}{2}}(\sqrt{c}|y|)(^\gamma\mathbf{T}^y_x\varphi)(x)y^\gamma dy.
	\end{equation}
\end{theorem}
\begin{proof}
 Let $\alpha{\,>\,}0$. Applying the formulas \eqref{BF03} and \eqref{SdExp}, we obtain
\begin{multline*}	
	(cI-\Delta_\gamma)^{-\frac{\alpha}{2}} \varphi(x)=\frac{1}{\Gamma\left(\frac{\alpha}{2}\right) }
	\int\limits_0^\infty t^{\frac{\alpha}{2}-1}(S^\gamma_t\varphi)(x)dt=\frac{1}{\Gamma\left(\frac{\alpha}{2}\right) }
	\int\limits_0^\infty t^{\frac{\alpha}{2}-1}([\mathbf{F}_\gamma^{-1}([\mathbf{F}_\gamma\varphi](\xi) e^{-t(c+\lvert\xi\rvert^2)})](x)dt=\\
	=\frac{1}{\Gamma\left(\frac{\alpha}{2}\right) }\frac{2^{\frac{n-|\gamma|}{2}}}{\prod\limits_{j=1}^n\
		\Gamma\left(\frac{\gamma_j{+}1}{2}\right)}\int\limits_0^\infty t^{\frac{\alpha}{2}-1}\left(\,\,\int\limits_{\mathbb{R}^n_+}e^{-t(c+\lvert\xi\rvert^2)}  [\mathbf{F}_\gamma\varphi](\xi)\mathbf{j}_\gamma(x,\xi) \xi^\gamma\:d\xi\right)dt=\\
		=\frac{1}{\Gamma\left(\frac{\alpha}{2}\right) }\frac{2^{\frac{n-|\gamma|}{2}}}{\prod\limits_{j=1}^n\
			\Gamma\left(\frac{\gamma_j{+}1}{2}\right)}\int\limits_{\mathbb{R}^n_+}\left(\,\,\int\limits_0^\infty t^{\frac{\alpha}{2}-1}e^{-t(c+\lvert\xi\rvert^2)}dt\right)  [\mathbf{F}_\gamma\varphi](\xi)\mathbf{j}_\gamma(x,\xi) \xi^\gamma\:d\xi=\\
				=\frac{2^{\frac{n-|\gamma|}{2}}}{\prod\limits_{j=1}^n\
				\Gamma\left(\frac{\gamma_j{+}1}{2}\right)}\int\limits_{\mathbb{R}^n_+} \left(c+|\xi|^2\right)^{-\frac{\alpha}{2}}  [\mathbf{F}_\gamma\varphi](\xi)\mathbf{j}_\gamma(x,\xi) \xi^\gamma\:d\xi\leq\\
						\leq \frac{2^{\frac{n-|\gamma|}{2}}}{\prod\limits_{j=1}^n\
							\Gamma\left(\frac{\gamma_j{+}1}{2}\right)}\int\limits_{\mathbb{R}^n_+} \left(c+|\xi|^2\right)^{-\frac{\alpha}{2}}  |[\mathbf{F}_\gamma\varphi](\xi)|\cdot |\mathbf{j}_\gamma(x,\xi)| \xi^\gamma\:d\xi<\infty.
\end{multline*}		
Thus, the integral defining the operator $(cI-\Delta_\gamma)^\alpha $ converges for $0{\,<\,}\alpha$.

 Applying \eqref{BF03} and \eqref{ConvS}, we obtain
\begin{multline*} 
	(cI-\Delta_\gamma)^{-\frac{\alpha}{2}} \varphi(x)=\frac{1}{\Gamma\left(\frac{\alpha}{2}\right)}\int\limits_0^\infty t^{\frac{\alpha}{2}-1} (S^\gamma_t\varphi)(x)dt
	=\frac{1}{\Gamma\left(\frac{\alpha}{2}\right) }\int\limits_0^\infty t^{\frac{\alpha}{2}-1}  \left(\,\,\int\limits_{\mathbb{R}^n_+}W_\gamma(y,t) (^\gamma\mathbf{T}^y_x\varphi)(x)y^\gamma dy \right) dt=\\
	=\frac{1}{\Gamma\left(\frac{\alpha}{2}\right) }\int\limits_{\mathbb{R}^n_+}\left(\,\, \int\limits_0^\infty t^{\frac{\alpha}{2}-1} W_\gamma(y,t)dt \right) (^\gamma\mathbf{T}^y_x\varphi)(x)y^\gamma dy.
\end{multline*}		
Let us calculate the integral by $t$ using  formula 2.3.16.1 from \cite{IR1}  of the form
$$
\int\limits_0^\infty x^{\beta-1}e^{-px-\frac{q}{x}}dx=2\left(\frac{q}{p} \right)^{\frac{\beta}{2}} K_{\beta}(2\sqrt{pq}).
$$
We obtain
\begin{multline*} 
 \int\limits_0^\infty t^{\frac{\alpha}{2}-1} W_\gamma(y,t)dt=\frac{1}{2^{|\gamma| }\prod\limits_{i=1}^n{\Gamma\left(\frac{\gamma_i{+}1}{2}\right)}} \int\limits_0^\infty  t^{\frac{\alpha}{2}-\frac{n+|\gamma| }{2}-1}{e^{-\frac{|y|^2}{4t}-ct}}dt=\\
 =\frac{1}{2^{|\gamma|-1} \prod\limits_{i=1}^n{\Gamma\left(\frac{\gamma_i{+}1}{2}\right)}}
 \left(\frac{|y|}{2\sqrt{c}} \right)^{\frac{\alpha-n-|\gamma|}{2}} K_{\frac{\alpha-n-|\gamma|}{2}}(\sqrt{c}|y|),
\end{multline*}	
 that gives \eqref{FrL01}.
\end{proof}

For $c{\,=\,}1$ operator \eqref{FrL01} coincides with the {\bf generalized Bessel potential} studied in  \cite{ShishkinaEkinciogluKeskin,Dzhabrailov}.
Explicit inversion formulas of Balakrishnan--Rubin type and a characterization of Bessel potentials associated with the Laplace--Bessel differential operator
$$
\Delta_B = \sum_{k=1}^{n-1} \frac{\partial^2}{\partial x_k^2}
+ \left( \frac{\partial^2}{\partial x_n^2} + \frac{2\nu}{x_n} \frac{\partial}{\partial x_n} \right)
\quad (\nu > 0)
$$
were obtained in \cite{Aliev}.

There are some properties of $(cI-\Delta_\gamma)^{-\frac{\alpha}{2}} \varphi$  are known (see \cite{Dzhabrailov} for the case $c{\,=\,}1$). 
Namely,
\begin{enumerate}
	\item[(1)] For every $\alpha{\,>\,}0$, the operator $(cI-\Delta_\gamma)^{-\frac{\alpha}{2}}$ maps $L^p_\gamma({\mathbb{R}}^{n}_+)$ into itself and is bounded with respect to the norm $\|\cdot\|_{1,\gamma}$.
	
	\item[(2)] For every function $\varphi{\,\in\,} L^p_\gamma(\mathbb{R}^n_+)$, where $1{\,<\,}p{\,<\,}q{\,<\,}\infty$ and $\frac{1}{q}{\,=\,}\frac{1}{p}{\,-\,}\frac{\alpha}{n+|\gamma|}$ with $0{\,<\,}\alpha{\,<\,}n{\,+\,}|\gamma|$, there exists a constant $C{\,=\,}C(n,\gamma,\alpha,p){\,<\,}\infty$ such that
	$$
	\|(cI-\Delta_\gamma)^{-\frac{\alpha}{2}} \varphi\|_{q,\gamma} \leq C \|\varphi\|_{p,\gamma}.
	$$
	In other words, $(cI-\Delta_\gamma)^{-\frac{\alpha}{2}} \varphi \in L^\gamma_q(\mathbb{R}^n_+)$.
\end{enumerate}

Reversing the sign of $\alpha$ in formula \eqref{FrL01} is not possible, since the resulting integral
has a strong singularity (of power order greater than $n+|\gamma|$) and diverges for bounded $\varphi$. 

Let us consider
$$
\omega_{\alpha,\gamma}(\sqrt{c}|x|)=\frac{2^{\frac{n-|\gamma|-\alpha}{2}+1} }
{\Gamma\left(\frac{\alpha}{2} \right)\prod\limits^n_{i=1}\Gamma\left(\frac{\gamma_i+1}{2}\right)}\,(\sqrt{c}|x|)^{\frac{n+|\gamma|-\alpha}{2}}\,K_{\frac{n+|\gamma|-\alpha}{2}}(\sqrt{c}|x|).
$$

Then the generalized Bessel potential \eqref{FrL01} can be represented as a generalized convolution operator in the form
$$
	(cI-\Delta_\gamma)^{-\frac{\alpha}{2}}\varphi=\left(\frac{\omega_{\alpha,\gamma}(\sqrt{c}|x|)}{(\sqrt{c}|x|)^{n+|\gamma|-\alpha}}*\varphi\right)_\gamma,\qquad \alpha>0.
$$

Let us consider the convolution operator 
$$
(cI-\Delta_\gamma)^{\frac{\alpha}{2}}_\varepsilon\varphi=( g^\alpha_{\gamma,\varepsilon,c}*\varphi)_\gamma
$$
with the kernel
$$
	g^\alpha_{\gamma,\varepsilon,c}(x)= \left( \mathbf{F}_\gamma^{-1}(c+|\xi|^2)^{\frac{\alpha}{2}}\cdot e^{-\varepsilon|\xi|}\right) (x).
$$
  For  $\varphi\in S_{ev}(\mathbb{R}^n_+)$ operator 
 $(cI-\Delta_\gamma)^{\frac{\alpha}{2}}_\varepsilon\varphi$ is
 bounded in $L^p_\gamma(\mathbb{R}^n_+)$, $1{\,<\,}p{\,<\,}+\infty$ (see proof in \cite{DzhabrailovLuchkoS} for $c{\,=\,}1$).

We define the positive fractional power of $(I-\Delta_\gamma)$ by passing to the limit
\begin{equation}\label{PF1111}
(I-\Delta_\gamma)^{\frac{\alpha}{2}}\varphi(x)
	=
	\lim_{\varepsilon \to 0} \,
	\left(
	\mathbf{F}_\gamma^{-1}\big[(1+|\xi|^2)^{\frac{\alpha}{2}} e^{-\varepsilon|\xi|}\big](x) * \varphi(x)
	\right)_\gamma,
\end{equation}
where the limit is taken in $L_p^\gamma(\mathbb{R}^n_+)$.

In \cite{DzhabrailovLuchkoS}, an alternative representation for $(I-\Delta_\gamma)^{\frac{\alpha}{2}}$ is derived using the Taylor-Delsarte formula.

\section{Generalised Bessel potential spaces}

\subsection{Definition and norm}

The purpose of this paragraph is to  study the scale of generalized Bessel potential spaces, which generalize both the classical H\"{o}lder spaces and the Sobolev spaces.

Let $\alpha{\,\in\,}\mathbb{R}$, $1{\,\leq\,}p{\,\leq\,}\infty$.
 Weighted Sobolev space with non-integer order $\alpha$ is defined as
$$
H^{\alpha,p}_\gamma(\mathbb{R}^n_+) =\left\{ 
f \in \mathcal{S}'_{ev}(\mathbb{R}^n_+):(I-\Delta_\gamma)^{\frac{\alpha}{2}} f
  \in L^p_\gamma(\mathbb{R}^n_+)
\right\}.
$$
Depending on the value of $\alpha$, the operator $(I-\Delta_\gamma)^{\frac{\alpha}{2}}$ is defined as follows: for $\alpha{\,<\,} 0$, formula \eqref{FrL01} applies; for $\alpha{\,>\,}0$, formula \eqref{PF1111} applies; and for $\alpha{\,=\,}0$, it reduces to the identity operator $I$.
Also, when 	$\alpha{\,<\,} 0$ 
we can write
$(I-\Delta_\gamma)^{\frac{\alpha}{2}}\varphi{\,=\,}\mathbf{F}_\gamma^{-1}(1+|x|^2)^{\frac{\alpha}{2}} \mathbf{F}_\gamma\varphi$.

Space $H^{\alpha,p}_\gamma(\mathbb{R}^n_+)$
is called {\bf generalised Bessel potential space}. 

By definition $(I-\Delta_\gamma)^{\frac{\alpha}{2}}$ is an isomorphism from $H^{\alpha,p}_\gamma(\mathbb{R}^n_+)$ to $L^p_\gamma(\mathbb{R}^n_+)$.

For $1{\,\leq\,}p{\,<\,}\infty$, the norm in $H^{\alpha,p}_\gamma(\mathbb{R}^n_+)$ is defined by
$$
\|f\|_{H^{\alpha,p}_\gamma(\mathbb{R}^n_+)}= 
\left\| 
(I-\Delta_\gamma)^{\frac{\alpha}{2}}f
\right\|_{L^p_\gamma(\mathbb{R}^n_+)}.
$$

	Explicitly, this means
$$
	\|f\|_{H^{\alpha,p}_\gamma(\mathbb{R}^n_+)}
	=
	\left(\,\, 
	\int\limits_{\mathbb{R}^n_+} 
	\left| 
	(I-\Delta_\gamma)^{\frac{\alpha}{2}}f(x)
	\right|^p 
	\, x^\gamma dx
	\right)^{1/p}.
	$$
	
	For $p{\,=\,}\infty$, the norm is given by the essential supremum
	$$
	\|f\|_{H^{\alpha,\infty}_\gamma(\mathbb{R}^n_+)}
	=
\underset{x \in \mathbb{R}^n_+}{\operatorname{ess\,sup}}
	\left| 
	(I-\Delta_\gamma)^{\frac{\alpha}{2}}f(x)
	\right|.
	$$
	
$H^{\alpha,p}_\gamma(\mathbb{R}^n_+)$ are Banach spaces in general.
 In the special case $p{\,=\,}2$ 
\begin{equation}\label{Skpr}
 \langle u,v \rangle_{H^{\alpha,2}_\gamma(\mathbb{R}^n_+)}=\left\langle (I-\Delta_\gamma)^{\frac{\alpha}{2}} u, 	(I-\Delta_\gamma)^{\frac{\alpha}{2}} v
 \right\rangle_{L_\gamma^2(\mathbb{R}_+^n)}
\end{equation}
is an inner product on $H^{\alpha,2}_\gamma(\mathbb{R}^n_+)$ and $\|u\|_{H^{\alpha,2}_\gamma(\mathbb{R}^n_+)}^2=\langle u,u \rangle_{H^{\alpha,2}_\gamma(\mathbb{R}^n_+)}$.
 Thus $H^{\alpha,2}_\gamma(\mathbb{R}^n_+)$ is a Hilbert space.

	When $p{\,=\,}2$, the space $H^{\alpha,2}_\gamma(\mathbb{R}^n_+)$ is  denoted simply as $H^\alpha_\gamma(\mathbb{R}^n_+)$. In this special case, by Plancherel's identity for Hankel transform \eqref{Plancherel}, the norm simplifies for $\alpha{\,<\,}0$ to
	$$
	\|f\|_{H^{\alpha,2}_\gamma(\mathbb{R}^n_+)}=\|f\|_{H^{\alpha}_\gamma(\mathbb{R}^n_+)}
	=
	\left( \,\,
	\int\limits_{\mathbb{R}^n_+} 
	(1 + |\xi|^2)^\frac{\alpha}{2} \, |\mathbf{F}_\gamma f(\xi)|^2 \, \xi^\gamma d\xi
	\right)^{1/2}.
	$$

\subsection{Reproducing Kernel for the Hilbert Space of Generalised Bessel Potentials}

The generalised Bessel potential space   $H^{\alpha}_\gamma(\mathbb{R}^n_+) $ is a reproducing kernel Hilbert space (RKHS) if and only if the fractional order of differentiability satisfies $\alpha{\,>\,}\frac{n+|\gamma|}{2}$. 
In other words, the reproducing kernel property holds exactly when the smoothness parameter exceeds half the $n{\,+\,}|\gamma|$.	

Let $\alpha{\,>\,}0$ and
\begin{equation}\label{kernel}
	G_{\alpha}^\gamma(x)=	\mathbf{F}_\gamma^{-1}[(1+|\xi|^2)^{-\frac{\alpha}{2}}](x)=
	\frac{2^{\frac{n-|\gamma|-\alpha}{2}+1}}{\Gamma\left(\frac{\alpha}{2}\right)\prod\limits_{i=1}^n{\Gamma\left(\frac{\gamma_i{+}1}{2}\right)}} |x|^{\frac{\alpha-n-|\gamma|}{2}}K_{\frac{n+|\gamma|-\alpha}{2}}(|x|),
\end{equation}	
then
$$
(I-\Delta_\gamma)^{-\frac{\alpha}{2}} \varphi(x)=\int\limits_{\mathbb{R}^n_+}G_{\alpha}^\gamma(y)(^\gamma\mathbf{T}^y_x\varphi)(x)y^\gamma dy=
\int\limits_{\mathbb{R}^n_+}(^\gamma\mathbf{T}^y_xG_{\alpha}^\gamma(x))\varphi(y)y^\gamma dy=\left\langle \,^\gamma\mathbf{T}^yG_{\alpha}^\gamma,\varphi\right\rangle_\gamma
$$
or
$$
(I-\Delta_\gamma)^{-\frac{\alpha}{2}} \varphi=(G_{\alpha}^\gamma*\varphi)_\gamma.
$$

\begin{theorem}\label{T2}
For $\alpha{\,>\,}0$ the space $H^{\alpha}_\gamma(\mathbb{R}^n_+)$ is a reproducing kernel Hilbert space  equipped 
with the  norm 
$$
\|f\|_{H^{\alpha}_\gamma(\mathbb{R}^n_+)}
=
\left(\,\, 
\int\limits_{\mathbb{R}^n_+} 
\left| 
(I-\Delta_\gamma)^{\frac{\alpha}{2}}f(x)
\right|^2 
\, x^\gamma dx
\right)^{1/2}
$$
and function $\,^\gamma\mathbf{T}^yG_{2\alpha}^\gamma$ is the reproducing kernel for this Hilbert space.
\end{theorem}
\begin{proof}
First, we need to show that	for every $y \in \mathbb{R}^n_+$, the function ${}^\gamma\mathbf{T}^y G_{2\alpha}^\gamma$ belongs to $H^{\alpha}_\gamma(\mathbb{R}^n_+)$. Consequently, it suffices to prove that $(I-\Delta_\gamma)^{\frac{\alpha}{2}} \, {}^\gamma\mathbf{T}^y G_{2\alpha}^\gamma \in L^2_\gamma(\mathbb{R}^n_+)$.
By \eqref{Plancherel} we can write
$$
\|(I-\Delta_\gamma)^{\frac{\alpha}{2}} \, {}^\gamma\mathbf{T}^y G_{2\alpha}^\gamma\|_{L^2_\gamma(\mathbb{R}^n_+)}^2{=}	\int\limits_{\mathbb{R}_+^n} |(I-\Delta_\gamma)^{\frac{\alpha}{2}} \, {}^\gamma\mathbf{T}^y G_{2\alpha}^\gamma|^2 \, x^\gamma \, dx
	{=}	\int\limits_{\mathbb{R}_+^n} |\mathbf{F}_\gamma[(I-\Delta_\gamma)^{\frac{\alpha}{2}} \, {}^\gamma\mathbf{T}^y G_{2\alpha}^\gamma](\eta)|^2 \, \eta^\gamma \, d\eta.
$$
Writing the formula \eqref{PF1111} explicitly before passing to the limit yields
$$
\mathbf{F}_\gamma\left[	\left(
\mathbf{F}_\gamma^{-1}\big[(1+|\xi|^2)^{\frac{\alpha}{2}} e^{-\varepsilon|\xi|}\big](y) * {}^\gamma\mathbf{T}^y_x G_{2\alpha}^\gamma(x)
\right)_\gamma\right](\eta)=(1+|\eta|^2)^{-\frac{\alpha}{2}} e^{-\varepsilon|\eta|}\mathbf{j}_{\gamma}(x,\eta).
$$	
Then,
\begin{multline*}
\|(I-\Delta_\gamma)^{\frac{\alpha}{2}} \, {}^\gamma\mathbf{T}^y G_{2\alpha}^\gamma\|_{L^2_\gamma(\mathbb{R}^n_+)}^2=\lim\limits_{\varepsilon\to 0}\int\limits_{\mathbb{R}_+^n} \left|(1+|\eta|^2)^{-\frac{\alpha}{2}} e^{-\varepsilon|\eta|}\mathbf{j}_{\gamma}(x,\eta)\right|^2 \eta^\gamma \, d\eta=\\
=\lim\limits_{\varepsilon\to 0}\int\limits_{\mathbb{R}_+^n} (1+|\eta|^2)^{- \alpha} e^{-2\varepsilon|\eta|}\mathbf{j}_{\gamma}^2(x,\eta)  \eta^\gamma \, d\eta=\{\eta=r\theta\}=\\
=\lim\limits_{\varepsilon\to 0}\int\limits_{0}^\infty (1+r^2)^{- \alpha} e^{-2\varepsilon r}r^{n+|\gamma|-1}dr\int\limits_{S^+_1(n)}\mathbf{j}_{\gamma}^2(x,r\theta)  \theta^\gamma \, dS=\\ 
=C\lim\limits_{\varepsilon\to 0}\int\limits_{0}^\infty (1+r^2)^{- \alpha} e^{-2\varepsilon r}r^{n+|\gamma|-1}dr<\infty.
\end{multline*}
 Here $\theta^\gamma{=}\prod\limits_{i=1}^{n}\theta_i^{\gamma_i},$ $S^+_1(n){=}\{\theta{:}|\theta|{=}1,\theta{\in}\mathbb{R}^n_+\}$ is a part of a sphere in
 $\mathbb{R}^n_+$.
	
It remains to show that, for every $x{\,\in\,}\mathbb{R}^n_+$ and every $\varphi{\,\in\,}H^{\alpha}_\gamma(\mathbb{R}^n_+)$, the evaluation functional at $x$ is represented by the inner product
	$$
	\varphi(x) = \langle \varphi,\,^\gamma\mathbf{T}^yG_{\alpha}^\gamma \rangle_{H^{\alpha}_\gamma(\mathbb{R}^n_+)}.
	$$
Indeed, by \eqref{Parseval} and  \eqref{Skpr} we get
\begin{multline*}
 \langle \varphi,\,^\gamma\mathbf{T}^yG_{2\alpha}^\gamma \rangle_{H^{\alpha}_\gamma(\mathbb{R}^n_+)}=\left\langle (I-\Delta_\gamma)^{\frac{\alpha}{2}} \varphi, 	(I-\Delta_\gamma)^{\frac{\alpha}{2}} \,^\gamma\mathbf{T}^yG_{2\alpha}^\gamma \right\rangle_{L_\gamma^2(\mathbb{R}_+^n)}=\\
=\int\limits_{\mathbb{R}^n_+} (I-\Delta_\gamma)^{\frac{\alpha}{2}} \varphi(y)\cdot (I-\Delta_\gamma)^{\frac{\alpha}{2}} \,^\gamma\mathbf{T}^y_xG_{2\alpha}^\gamma(x)\cdot y^\gamma dy=\\
=\lim_{\varepsilon \to 0} \int\limits_{\mathbb{R}^n_+} \mathbf{F}_\gamma \left[\left(
\mathbf{F}_\gamma^{-1}\big[(1+|\xi|^2)^{\frac{\alpha}{2}} e^{-\varepsilon|\xi|}\big](y) * \varphi(y)
\right)_\gamma\right](\eta)\cdot  \mathbf{F}_\gamma \left[(I-\Delta_\gamma)^{\frac{\alpha}{2}}\,^\gamma\mathbf{T}^y_xG_{2\alpha}^\gamma(x)\right](\eta)\cdot \eta^\gamma d\eta.
\end{multline*}

Applying \eqref{HanSdv} and \eqref{fursver1}, we obtain
\begin{equation*}
 \mathbf{F}_\gamma \left[\left(
\mathbf{F}_\gamma^{-1}\big[(1+|\xi|^2)^{\frac{\alpha}{2}} e^{-\varepsilon|\xi|}\big](y) * \varphi(y)
\right)_\gamma\right](\eta)
=(1+|\eta|^2)^{\frac{\alpha}{2}} e^{-\varepsilon|\eta|}\cdot \mathbf{F}_\gamma[\varphi](\eta)
\end{equation*}
and
\begin{multline*}
\mathbf{F}_\gamma[(I-\Delta_\gamma)^{\frac{\alpha}{2}} \,^\gamma\mathbf{T}^yG_{2\alpha}^\gamma](\eta)=\lim_{\varepsilon_0 \to 0} \,
\mathbf{F}_\gamma\left[\left(
\mathbf{F}_\gamma^{-1}\big[(1+|\xi|^2)^{\frac{\alpha}{2}} e^{-\varepsilon_0|\xi|}\big](y) * \,^\gamma\mathbf{T}^y_xG_{2\alpha}^\gamma(x)
\right)_\gamma\right](\eta)=\\
=\lim_{\varepsilon_0 \to 0} \,(1+|\eta|^2)^{\frac{\alpha}{2}} e^{-\varepsilon_0|\eta|}\mathbf{F}_\gamma\left[\,^\gamma\mathbf{T}^\eta_xG_{2\alpha}^\gamma(x)\right](\eta)=\lim_{\varepsilon_0 \to 0} \,(1+|\eta|^2)^{\frac{\alpha}{2}} e^{-\varepsilon_0|\eta|}\mathbf{j}_{\gamma}(x,\eta)\,\mathbf{F}_\gamma\left[G_{2\alpha}^\gamma\right](\eta)=\\
=\lim_{\varepsilon_0 \to 0} \,(1+|\eta|^2)^{\frac{\alpha}{2}} e^{-\varepsilon_0|\eta|}\mathbf{j}_{\gamma}(x,\eta)\,(1+|\eta|^2)^{-\alpha}=(1+|\eta|^2)^{-\frac{\alpha}{2}} \mathbf{j}_{\gamma}(x,\eta).
\end{multline*}
Therefore,
\begin{multline*}
  \langle \varphi,\,^\gamma\mathbf{T}^yG_{\alpha}^\gamma \rangle_{H^{\alpha}_\gamma(\mathbb{R}^n_+)}=\lim_{\varepsilon \to 0}\int\limits_{\mathbb{R}^n_+} (1+|\eta|^2)^{\frac{\alpha}{2}} e^{-\varepsilon|\eta|}\cdot \mathbf{F}_\gamma[\varphi](\eta)\cdot (1+|\eta|^2)^{-\frac{\alpha}{2}}
  \mathbf{j}_{\gamma}(x,\eta) \eta^\gamma d\eta=\\
  = \mathbf{F}_\gamma^{-1}[ \mathbf{F}_\gamma[\varphi](\eta)](x)=\varphi(x).
\end{multline*}
 \end{proof}

	\begin{corollary}
	Function $\,^\gamma\mathbf{T}^yG_{2\alpha}^\gamma$  for $\alpha{\,>\,}0$ is a Green function	 for the operator  
	$(I-\Delta_\gamma)^{\alpha}$ in the sense  
	$$
	\int\limits_{\mathbb{R}^n_+} (I-\Delta_\gamma)^{\alpha}\,^\gamma\mathbf{T}^y_xG_{2\alpha}^\gamma(x)\, \varphi(y)y^\gamma dy=\varphi(x).
	$$
\end{corollary}

\section{Singular Poisson's Equation}

In \cite{ByunSaitoh,Saitoh1997,SaitohSawano}, it was demonstrated that the theory of reproducing kernels, together with Tikhonov regularization, provides a simple and practical framework for obtaining approximate solutions to diverse inverse problems. Here we illustrate the RKHS method by applying it to approximate the solution of the inverse problem for the singular Poisson equation.

The next theorem from \cite{Matsuura} gives instructions on how to construct an RKHS for a given operator.

\begin{theorem}\label{T3}
	Let $\mathcal{H}_K$ be a Hilbert space admitting the reproducing kernel $K(x,y)$ on a set $E$. Let $L{\,:\,}\mathcal{H}_K \to \mathcal{H}$ be a bounded linear operator from $\mathcal{H}_K$ into $\mathcal{H}$. For $\lambda{\,>\,}0$, introduce the inner product in $\mathcal{H}_K$ and call it $\mathcal{H}_{K_\lambda}$ as
	$$
	\langle u, v \rangle_{\mathcal{H}_{K_\lambda}}
	= \lambda \langle u, v \rangle_{\mathcal{H}_K}
	+ \langle L u, L v \rangle_{\mathcal{H}}, 
	$$
	then $\mathcal{H}_{K_\lambda}$ is the Hilbert space with the reproducing kernel $K_\lambda(p, q)$ on $E$ and satisfying the equation
$$
	K(\cdot, y) = (\lambda I + L^* L) K_\lambda(\cdot, y), 
$$
	where $L^*$ is the adjoint of $L{\,:\,}\mathcal{H}_K \to \mathcal{H}$.
\end{theorem}

The singular Poisson equation is
\begin{equation}\label{Pois}
\Delta_\gamma u=g.
\end{equation}
Here we suppose $g{\,\in\,}L^2_\gamma(\mathbb{R}^n_+)$.

It is known (see \cite{LyakhovRiesz,LyakhovHyp}) that the Riesz $B$-potential  
$$u(x)=-\frac{\Gamma\left(\frac{n+|\gamma|}{2} - 1\right)}{2^{2-n} \prod\limits^n_{i=1}\Gamma\left(\frac{\gamma_i+1}{2}\right)}
\int\limits_{\mathbb{R}^n_+}|y|^{{2-n-|\gamma|}}\,^\gamma\mathbf{T}_x^y g(x)y^\gamma dy
$$
gives the solution to the singular Poisson equation \eqref{Pois}.

For example, from \eqref{DelBes} it is easy to see that 	
solution to the equation $\Delta_\gamma u{\,=\,}\mathbf{j}_\gamma(x,\xi)$ is $u(x){\,=\,}{-}\frac{\mathbf{j}_\gamma(x;\xi)}{|\xi|^2}$.
However, in general, calculations with the Riesz $B$-potential are not straightforward, so we provide a simple approximate solution to the singular Poisson equation in a specially constructed RKHS.

Using Theorems \ref{T2} and \ref{T3} let us construct RKHS for the Laplace-Bessel operator. 
First, we obtain the result about boundedness of the Laplace-Bessel operator in the space $H^{\alpha}_\gamma(\mathbb{R}^n_+)$.

\begin{lemma}\label{Lem1}
	The Laplace-Bessel operator $\Delta_\gamma$  is a bounded linear operator from the generalised Bessel potential space $H^{\alpha}_\gamma(\mathbb{R}^n_+)$ to the $L^2_\gamma(\mathbb{R}^n_+)$ if and only if $\alpha\geq 2$.
\end{lemma} 
\begin{proof}
Let us prove the inequality
	$$
	||\Delta_\gamma u||^2_{L^2_\gamma(\mathbb{R}^n_+)}\leq ||u||^2_{H^\alpha_\gamma(\mathbb{R}^n_+)},
	$$
	that is, the operator $\Delta_\gamma$ is a bounded linear operator from $H^\alpha_\gamma(\mathbb{R}^n_+)$ into $L^2_\gamma(\mathbb{R}^n_+)$.  
	
		Recall the Plancherel's identity \eqref{Plancherel}  for $L^2_\gamma(\mathbb{R}^n_+)$ and formula \eqref{LBH}, we obtain
$$
||\Delta_\gamma u||^2_{L^2_\gamma(\mathbb{R}^n_+)}=\int\limits_{\mathbb{R}_+^n} |\Delta_\gamma u(x)|^2 \, x^\gamma \, dx
	=
	\int\limits_{\mathbb{R}_+^n} |\mathbf{F}_\gamma[\Delta_\gamma u(x)](\xi)|^2 \, \xi^\gamma \, d\xi=	\int\limits_{\mathbb{R}_+^n} |\xi|^4\cdot |\mathbf{F}_\gamma[u](\xi)|^2 \, \xi^\gamma \, d\xi.
$$
	
If and only if $\alpha{\,\ge\,}2$, then for every $\xi{\,\in\,} \mathbb{R}^n_+$ the following inequality holds
	$$
		|\xi|^4 \leq (1 + |\xi|^2)^\alpha.
	$$
Thus,
	$$
	||\Delta_\gamma u||^2_{L^2_\gamma(\mathbb{R}^n_+)}
			\leq \int\limits_{\mathbb{R}^n_+} (1 + |\xi|^2)^\alpha |\mathbf{F}_\gamma[u](\xi)|^2 \, d\xi
		= \|u\|_{H^\alpha_\gamma(\mathbb{R}^n_+)}^2,
	$$
		which proves that $\Delta_\gamma{\,:\,}H^\alpha_\gamma(\mathbb{R}^n_+) \to L^2_\gamma(\mathbb{R}^n_+)$ is bounded if and only if $\alpha{\,\geq\,}2$.
\end{proof}
 
 We will consider functions from $H^\alpha_\gamma(\mathbb{R}^n_+)$, $\alpha{\,\geq\,}2$. In this case Laplace-Bessel operator is self-adjoint.
 Now we can introduce the inner product
\begin{equation}\label{InPr}
\langle u, v \rangle_{H^\alpha_{\gamma,\lambda}(\mathbb{R}^n_+)}
= \lambda \langle u, v \rangle_{H^\alpha_\gamma(\mathbb{R}^n_+)}
+ \langle \Delta_\gamma u, \Delta_\gamma v \rangle_{L^2_\gamma(\mathbb{R}^n_+)}, 
\end{equation}
the norm
$$
\|u \|_{H^\alpha_{\gamma,\lambda}(\mathbb{R}^n_+)}=\sqrt{\lambda\|u \|_{H^\alpha_{\gamma}(\mathbb{R}^n_+)}^2+\| \Delta_\gamma u\|_{L^2_\gamma(\mathbb{R}^n_+)}^2},
$$
and reproducing kernel
 $$
K^\gamma(\cdot, y) = (\lambda I + \Delta_\gamma^2) K_\lambda^\gamma(\cdot, y).
$$
From Theorem \ref{T2} we know that reproducing kernel in $H^\alpha_\gamma(\mathbb{R}^n_+)$ is $K^\gamma(\cdot, y){\,=\,}\,^\gamma\mathbf{T}^yG_{2\alpha}^\gamma$.
Therefore, in order to find $K_\lambda^\gamma(\cdot, y)$ we need to solve equation
$$
\,^\gamma\mathbf{T}^y_xG_{2\alpha}^\gamma(x) = (\lambda I + \Delta_\gamma^2) K_\lambda^\gamma(x, y).
$$
Applying the Hankel transform with respect to the variable $x$, and using \eqref{HanSdv}, \eqref{kernel} and \eqref{LBH}, we obtain
$$
\mathbf{F}_\gamma[\,^\gamma\mathbf{T}^y_xG_{2\alpha}^\gamma(x)](\xi) = \mathbf{F}_\gamma[(\lambda I + \Delta_\gamma^2) K_\lambda^\gamma(x, y)](\xi)\qquad \Rightarrow\qquad
$$
$$
\mathbf{j}_{\gamma}(y,\xi)\,(1+|\xi|^2)^{-\alpha}=(\lambda+|\xi|^4)\mathbf{F}_\gamma[K_\lambda^\gamma(x, y)](\xi)\qquad \Rightarrow\qquad
$$
$$
\mathbf{F}_\gamma[K_\lambda^\gamma(x, y)](\xi)=\frac{\mathbf{j}_{\gamma}(y,\xi)}{(1+|\xi|^2)^{\alpha}(\lambda+|\xi|^4)}\qquad \Rightarrow\qquad
$$
\begin{equation}\label{Ker}
K_\lambda^\gamma(x, y)=\frac{2^{\frac{n-|\gamma|}{2}}}{\prod\limits_{j=1}^n\,
	\Gamma\left(\frac{\gamma_j+1}{2}\right)}\int\limits_{\mathbb{R}^n_+}
\frac{\mathbf{j}_\gamma(x,\xi)\mathbf{j}_{\gamma}(y,\xi)}{(1+|\xi|^2)^{\alpha}(\lambda+|\xi|^4)}\xi^\gamma\:d\xi.
\end{equation}
 
Thus, we construct the RKHS $H^\alpha_{\gamma,\lambda}(\mathbb{R}^n_+)$ with inner product \eqref{InPr} and reproducing kernel \eqref{Ker}, which is adapted to the Laplace-Bessel operator $\Delta_\gamma$.
Taking the limit as $\lambda \to 0$ in $K_\lambda^\gamma(x, y)$, we recover the original kernel $K^\gamma(x, y)$:
$$
\lim_{\lambda \to 0} K_\lambda^\gamma(x, y) = K^\gamma(x, y).
$$

Tikhonov regularization is implemented in the following theorem from \cite{Matsuura}.

\begin{theorem}
Let $\mathcal{H}_K$, $L$, $\mathcal{H}$, $E$, and $K_\lambda$ be as in Theorem \ref{T3}. Then, for any $\lambda{\,>\,}0$ and for any $g{\,\in\,}\mathcal{H}$, the extremal function in
\begin{equation}\label{10}
\inf\limits_{f \in \mathcal{H}_K} \left\{ \lambda \| f \|_{\mathcal{H}_K}^2 + \| L f - g \|_{\mathcal{H}}^2 \right\}
\end{equation}
exists uniquely and the extremal function
\begin{equation}\label{11}
f_{\lambda, g}(p) = \langle g, L K_\lambda(\cdot, p) \rangle_{\mathcal{H}} 
\end{equation}
is the member of $\mathcal{H}_K$ which attains the infimum in \eqref{10}.
\end{theorem}

We are therefore ready to construct the Tikhonov regularized approximation of the solution to the singular Poisson equation as the best approximate solution in the space $H^\alpha_\gamma(\mathbb{R}^n_+)$.

\begin{theorem}
Let $n{\,\ge\,}1$ and $\alpha{\,\ge\,}2$. 
For any function $g{\,\in\,}L^2_\gamma(\mathbb{R}^n_+)$ and for any $\lambda{\,>\,}0$, the best approximate function 
$u_{\lambda,\alpha,g}^\gamma$ in the sense
$$
\inf_{u \in H^\alpha_\gamma(\mathbb{R}^n_+)} \left\{ \lambda \|u\|_{H^\alpha_\gamma(\mathbb{R}^n_+)}^2 + \|g - \Delta_\gamma u\|_{L^2_\gamma(\mathbb{R}^n_+)}^2 \right\}
= \lambda \|u_{\lambda,\alpha,g}^\gamma\|_{H^\alpha_\gamma}^2 + \|g - \Delta_\gamma u_{\lambda,\alpha,g}^\gamma\|_{L^2_\gamma(\mathbb{R}^n_+)}^2 
$$
exists uniquely and $F^{\delta}_{\delta, s, g}$ is represented by
\begin{equation}\label{44}
u_{\lambda,\alpha,g}^\gamma(x) = \int\limits_{\mathbb{R}^n_+} g(y) \,^\gamma \mathbf{T}_x^y Q_{\lambda, \alpha}^\gamma(x) \,y^\gamma dy 
\end{equation}
for
$$
Q_{\lambda,\alpha}^\gamma(x)=-\frac{2^{1-\frac{n+|\gamma|}{2}}}{\Gamma\left(\frac{n+|\gamma|}{2}\right)}  \int\limits_0^\infty 
\frac{j_{\frac{n+|\gamma|}{2}-1}(r|x|)}{(1+r^2)^{\alpha}(\lambda+r^4)}r^{n+|\gamma|+1}dr,
$$	
and as $\lambda{\,\to\,}0$,  $u_{\lambda,\alpha,g}^\gamma(x){\,\to\,}u(x)$ uniformly.
\end{theorem}
\begin{proof}
In Lemma \ref{Lem1}, it was proved that $\Delta_\gamma : H^\alpha_\gamma(\mathbb{R}^n_+) \to L^2_\gamma(\mathbb{R}^n_+)$ is bounded if and only if $\alpha \ge 2$. The kernel $K_\lambda^\gamma(x, y)$ was obtained in \eqref{Ker} and is given by
	$$
	 K_\lambda^\gamma(x, y)=\frac{2^{\frac{n-|\gamma|}{2}}}{\prod\limits_{j=1}^n\,
	 	\Gamma\left(\frac{\gamma_j+1}{2}\right)}\int\limits_{\mathbb{R}^n_+}
	 \frac{\,^\gamma\mathbf{T}^y_x\mathbf{j}_\gamma(x,\xi)}{(1+|\xi|^2)^{\alpha}(\lambda+|\xi|^4)}\xi^\gamma\:d\xi.
	$$
	So, we can  find $u_{\lambda,\alpha,g}^\gamma$ by formulas \eqref{11} and \eqref{JBes}:
\begin{multline*}
	u_{\lambda,\alpha,g}^\gamma(x)=\int\limits_{\mathbb{R}^n_+} g(y)(\Delta_\gamma)_y K_\lambda^\gamma(x, y)\cdot y^\gamma\:dy=\\
	=-\frac{2^{\frac{n-|\gamma|}{2}}}{\prod\limits_{j=1}^n\,
		\Gamma\left(\frac{\gamma_j+1}{2}\right)}\int\limits_{\mathbb{R}^n_+} g(y) \left(\,\, \int\limits_{\mathbb{R}^n_+}
		\frac{|\xi|^2\,^\gamma\mathbf{T}^y_x\mathbf{j}_\gamma(x,\xi)}{(1+|\xi|^2)^{\alpha}(\lambda+|\xi|^4)}\xi^\gamma\:d\xi \right) y^\gamma\:dy=\{\xi=r\theta\}=\\
			=-\frac{2^{\frac{n-|\gamma|}{2}}}{\prod\limits_{j=1}^n\,
				\Gamma\left(\frac{\gamma_j+1}{2}\right)}\int\limits_{\mathbb{R}^n_+} g(y)\,^\gamma\mathbf{T}^y_x \left(\,\, \int\limits_0^\infty 
				\frac{r^{n+|\gamma|+1}dr}{(1+r^2)^{\alpha}(\lambda+r^4)}
				\int\limits_{S_1^+(n)}
			\mathbf{j}_\gamma(x,r\theta)\theta^\gamma\:dS \right) y^\gamma\:dy=\\
			=-\frac{2^{1-\frac{n+|\gamma|}{2}}}{\Gamma\left(\frac{n+|\gamma|}{2}\right)}\int\limits_{\mathbb{R}^n_+} g(y)\,^\gamma\mathbf{T}^y_x \left(\,\, \int\limits_0^\infty 
			\frac{j_{\frac{n+|\gamma|}{2}-1}(r|x|)}{(1+r^2)^{\alpha}(\lambda+r^4)}r^{n+|\gamma|+1}dr
			  \right) y^\gamma\:dy.
\end{multline*}
That gives \eqref{44}.
\end{proof}

\section{acknowledgments}
This research was supported by the Russian Science Foundation (project No. 25-21-00830).

\end{document}